\documentclass[12pt,a4paper]{article}
 \usepackage[colorlinks]{hyperref}
\usepackage{amsmath}
\usepackage{amsthm}
\usepackage{amsfonts}
\usepackage{latexsym}
\newcounter{alphthm}
\newtheorem{theo}[alphthm]{Theorem}

\newtheorem{lemA}[alphthm]{Lemma}

\newtheorem{coro}[alphthm]{Corollary}
\newtheorem{thm}{Theorem}
\newtheorem{defn}{Definition}
\newtheorem{prop}{Proposition}
\newtheorem{cor}{Corollary}
\newtheorem{lem}{Lemma}

\newcommand{\be}{\begin{equation}}
\newcommand{\ee}{\end{equation}}
\newcommand{\bes}{\begin{equation*}}
\newcommand{\ees}{\end{equation*}}
\newcommand{\br}{\begin{remark}}
\newcommand{\er}{\end{remark}}
\newcommand{\ben}{\begin{enumerate}}
\newcommand{\een}{\end{enumerate}}
\newcommand{\pa}{{\partial}}

\newcommand{\pxi}{{\pa \over \pa x^i}}

\def\beq{\begin{equation}}
\def\eeq{\end{equation}}
\def\nn{\nonumber}
\newcommand{\R}{I\!\! R}

\title{On projective symmetries on Finsler Spaces}
\author{B. Lajmiri,\ \ B. Bidabad\thanks{The corresponding author, bidabad@aut.ac.ir;  behroz.bidabad@math.univ-toulouse.fr },\ \ M. Rafie-Rad,\ \ Y. Aryanejad-Keshavarzi }
\date{}
\begin{document}
 \maketitle
\begin{abstract}
  There are two definitions of Einstein-Finsler spaces introduced by Akbar-Zadeh, which we will show is equal along the integral curves of $I$-invariant projective vector fields.
  The sub-algebra of the $C$-projective vector fields, leaving the $H$-curvature invariant, has been studied extensively. Here we show on a closed Finsler space with negative definite Ricci curvature reduces to that of Killing vector fields.
 Moreover, if an Einstein-Finsler space admits such a projective vector field then the flag curvature is  constant.
Finally, a classification of compact isotropic mean Landsberg manifolds admitting certain projective vector fields is obtained with respect to the sign of Ricci curvature.
 \end{abstract}
 {{\bf{ Keywords}}:
Projective vector, symmetries,  the physical meaning of projective Lie algebra, Finsler, Landsberg, Isotropic mean Landsberg.\\
 \textbf{AMS subject classification} 53B40, 53C60.}
\section{Introduction}\setcounter{equation}{0}
The  class of isotropic mean Landsberg manifolds becomes recently  an omnipresent class of  metrics on both the theory and the applications of Finsler geometry. One of its essential roles in giving evidence to some ambiguous features of this geometry.
Without claiming to be exhaustive, we cite some of the more significant ones, see \cite{B2,BS2,BS3,M1,SB,S,TSP,Ti},  etc.

 The collection of all projective vector fields on a Finsler space $(M,F)$ is a finite-dimensional Lie algebra with respect to the usual Lie bracket,
called the projective algebra, and is denoted by $p(M,F)$. It is the Lie algebra of the projective group $P(M,F)$.
 Recently, the sub-algebra of the projective vector fields, leaving the Cartan curvature and the {\bf H}-curvature  invariant, has been studied. This sub-algebra contains vector fields leaving the mean Cartan tensor {\bf I} invariant.  We show that this sub-algebra on a closed Finsler space with negative definite Ricci curvature reduces to the sub-algebra of Killing vector fields.

 There are several Finsler metrics in this class, such as the
\textit{Funk metric} on strongly convex domains in ${\R^n}$. If $g_{ij}$ is the
fundamental tensor of a Finsler structure $F$, the vertical derivative of
$g_{ij}$ gives rise to the {\it Cartan tensor} {\bf
C}, and the horizontal derivative of $C$ along geodesics is called
the {\it {Landsberg tensor}} {\bf L}.

\subsection*{Why {\bf I}-invariant projective vector fields?}
As a {\bf physical motivation} for the study of projective symmetries recall that
the theory of relativity consists of two interdependent theories of Einstein: special relativity and general relativity. Special relativity applies to all physical phenomena in the absence of gravity. While general relativity explains the law of gravitation and its relation with other forces of nature. It includes astronomy, cosmology, and astrophysics fields.
Einstein's theory is currently considered to be the most adequate model of space-time and gravitation, of which Newtonian theory and the special theory of relativity are only valid approximations in particular circumstances. Its main goal is to elucidate, from a different point of view and in a systematic way, which mathematical facts and ideologies of facts are sufficient to support the edifice of the general relativity theory.
To this end, physicists can isolate certain relatively independent substructures that are contained in Einstein's theory, the topological, conformal, projective, affine, and  metric differential structures, which are related to different observable phenomena or  different relationships inferred from these phenomena.

In particular, the motion of particles in free fall defines a projective structure over space-time. This provides a projective connection, or an equivalence class of symmetrical affine connections all having the same non-parametric geodesic curve.

 This can be seen as a mathematical formulation of the weak principle of equivalence valid in both Newtonian and relativistic space-time and gravity theory, see \cite{Is}.
One of the current trends for many researchers is to seek Finslerian approaches to take advantage of the benefits of quantum technology and classical field theory.
 For example, so far, violations of relativity such as very special relativity 
 and very special general relativity 
  have been formulated.
  On the other hand, local anisotropy can be modeled by a nonlinear connection structure that generalizes that of affine connection.
    The metric-affine gravity contains various types of generalized Finsler-Lagrange-Hamilton-Cartan geometries distinguished by a metric and linear
connection adapted to the non-linear connection structure.

The mean Cartan tensor ${\bf I}$ can measure locally the geometric anisotropy of the underlying Finslerian space-time, as well as other non-Riemannian quantities ${\bf S}$, ${\bf E}$, ${\bf H}$, etc, (see Definitions\ref{Def;S} and \ref{Def;I,E,H}) are also referred to the local anisotropy.
However, they can only be physically significant if they are
invariant among the Finslerian cosmological metrics of space-time.
It's somehow similar to what Nurowski did in his work to make darkness an important issue in cosmology.
 Therefore, to have a consistent metric of local anisotropy, we should assume that it is a projective invariant quantity.
  On the other hand, based on the general covariance axiom in Physics, the quantities and tensors must have a uniform shape and be invariant under the observers' transformations. Mathematically, it means that we are only allowed to consider transformations that preserve all metric derived quantities such as the Cartan tensor ${\bf C}$ and the mean Cartan tensor ${\bf I}$.

  The study of projective symmetry is to consider projective vector fields that leave ${\bf I}$ invariant.

   In Proposition \ref{Prop;I-inv}, we elucidate the essential nature of the {\bf I}-symmetries.
 More intuitively, we show  ${\bf E}$ and ${\bf H}$ are projectively invariant, if {\bf I} is
projectively invariant.

Meanwhile, the following results are obtained in the process of this work.
\begin{thm}\label{Th;first}
Let $(M,F)$ be an Einstein Finsler space, satisfying ${ Ric_{ij}} = (n-1)kg_{ij},$ for some constant $k$. If $(M,F)$ admits a non-affine I-invariant projective vector field, then  it has the constant flag curvature ${k}$ at every point.
\end{thm}
\begin{thm}\label{prop1}\label{Th;second}
If a Landsberg space admits a non-affine {\bf I}-projective vector field, then it is Riemannian.
\end{thm}
\begin{thm}\label{Th;third}
Let $(M,F)$ be a compact isotropic mean Landsberg manifold and $X$ a projective
vector field leaving invariant the dual trace of torsion.
\begin{enumerate}
\item If $R(X,X)$ is negative definite, then the projective transformation
corresponding to $X$ reduces to identity,

\item If $R(X,X)$ is non-positive definite, then $X$ is Killing and its
horizontal
derivative vanishes, where $ R$, is the trace of Cartan hh-curvature tensor.
\end{enumerate}
\end{thm}
{\section{Preliminaries and notations}}\setcounter{equation}{0}
Let $M$ be an n-dimensional $ C^\infty$ manifold. Denote by $T_x M$ the tangent space at
$x \in M$, $TM:=\cup _{x \in M} T_xM $ the tangent bundle of $M$ and $TM_0 = TM
\setminus \{ 0 \}$, the slice tangent bundle. Each element of $TM$ has the form
$z=(x,y)$, where $x \in M$ and $y \in T_x M.$  The natural projection $\pi: TM
\rightarrow  M$ is given by $\pi (x,y):= x$. The \emph{pull-back tangent bundle} $\pi ^* TM$ is a vector bundle over
$TM_0$ and the fibers $\pi ^* _v TM$ at $v \in TM_0$ are a copy of $T_x M$,
where $\pi (v)=x$. We have therefore
$\pi ^* TM =\{(x,y,v)| y \in T_x M_0 , v\in T_x M
\}. $
\\
A \emph{Finsler
structure} on a differentiable manifold $M$ is a continuous function $F:TM\rightarrow [0,\infty)$, with the following properties; $F$ is differentiable on $TM_0:=TM\backslash\{0\}$ and  positively 1-homogeneous on the fibers of $TM$. The vertical Hessian of $F^{2}$ with the following components is
 positive-definite on $TM_0$,
 $(g_{ij}):=\left( \left[ {1 \over 2} F^2 \right]_{y^i y^j}\right).$

The Finsler structure  $F$ defines a fundamental tensor $g:\pi ^* TM \otimes \pi
^* TM \rightarrow [0,\infty)$, called the \emph{Finsler metric} with the components  $g(\partial_i
|_v,\partial_j |_{v})=g_{ij}(x,y)$, where $V=y^i {{\partial}\over {\partial x^i}}$ is a section of $\pi^\star TM$, and $v=V|_x=y^i {{\pa} \over {\partial x^i}}|_x$.
In this work a \emph{Finsler manifold} is denoted by the  pairs $(M,F)$  or
 $(M,g_{ij})$, unless another definition is explicitly stated. A Finsler metric reduces to a Riemannian one if the components $g_{ij} (x,y)$ are independent of $y \neq 0$.
 One can show easily
$
g_{ij}(x,y)=FF_{y^iy^j}+F_{y^i}F_{y^j},
$
where $F_{y^i}=\frac {\partial F }{\partial y^i}$. Hence $(\pi ^*
TM,g)$ becomes a Riemannian vector bundle over $TM_0$. Let us denote the components of \textit{Cartan tensor} by
$
 A_{ijk}(x,y):={1 \over 2}F(x,y)\dot{\partial }_k g_{ij} (x,y),
$
where $ \dot{\pa}_k:=\frac{\pa}{\pa y^{k}}$. The \emph{Cartan
 tensor} $A :\pi ^* TM \otimes \pi ^* TM \otimes\pi ^* TM \rightarrow
\mathbb{R} $ is defined by $A(\partial_i |_v ,\partial_j |_ {v},\partial_k |_ {v})=A_{ijk}(x,y),$
where $v=y^i {{\partial} \over {\partial x^i}}|_x$, and $A_{ijk}$ are symmetric with respect to $i,j,k$.
 By setting  $C^{i}_{jk}:=\frac{1}{2}g^{il}\dot{\partial}_{l}g_{jk},$
 we have ${C_{ijk}}= \frac{{{A_{ijk}}}}{F}$.
  The positive homogeneity
of $F$ and the Euler's theorem leads to $y^i{{\partial g_{ij}}\over {\partial y^k}}(x,y)=0.$  We recall that the components of the canonical
section $\ell:=\ell^i{{\pa  }\over {\pa x^i}}$ are defined by $
\ell^i(x,y):=\frac{y^i}{F(x,y)}.
$
Letting  $\ell_i:=g_{ij}\ell^j=F_{y^i}$, the canonical section
$\ell$ satisfies $g(\ell,\ell)=1$ and $A(X,Y,\ell)=0$.

Let $M$ be a connected smooth manifold and $(TM, \pi, M)$ its tangent bundle with no zero section.
 Every local chart on $M$ induces  a local chart on $TM$.
 The kernel of the linear
map $\pi_*:TTM \rightarrow TM$ is called the {\it vertical
distribution} and is denoted by $VTM$. For every $u \in TM$, $Ker\;\pi_{*,u}=V_uTM$ is spanned by \{${{\pa} \over {\pa y^i}}|_u$\}.
 A \emph{horizontal distribution}  $H:u \in TM \rightarrow H_uTM$
 is a complementary to the vertical distribution. We have thus the decomposition
$T_u(TM)=H_uTM \oplus V_uTM,\quad \forall u \in TM.$

%

Geodesics of a Finsler manifold $(M,{g_{ij}})$ are characterized by the  equations
$
\ddot{C}^i(t) + 2 G^i \Big (C(t), \dot{C}(t) \Big ) =0,
$
where, $\dot C(t) = \frac{{dC}}{{dt}}$, $\ddot C = \frac{{d\dot C}}{{dt}}$ and $G^i=G^i(x, y)$ are the spray coefficients given by
$ G^i = {1\over 4} g^{il}\Big \{ [F^2]_{x^ky^l} y^k - [F^2]_{x^l} \Big \}$, where here and everywhere in this paper the subscript indices $``x^k"$ and $``y^k"$  show the partial derivatives  $\frac{\pa}{\pa x^k}$ and  $\frac{\pa}{\pa y^k}$, respectively. If $F$ is Riemannian, then
$g_{ij}$ is a function of $x$ alone and $G^i(x,y) =
{1\over 2} \gamma^i_{jk}(x)y^jy^k$ are quadratic in $y= y^i\pxi|_x$ where $\gamma^i_{jk}(x)$ are ordinary Christoffel symbols.

One can observe that the pair $\{\frac{\delta}{\delta x^i},\frac{\pa}{\pa y^i}\}$ defined by $\frac{\delta}{\delta x^i}:=\frac{\pa}{\pa x^i}-G^j_i\frac{\pa}{\pa y^j}$, where  $ G^j_i:=\frac{\pa G^j}{\pa y^i}$  form a horizontal and vertical frame for $TTM$. The  horizontal and the vertical dual frame are given by the pair $\{{d x^i},{\delta y^i}\}$. We shall also refer to $\frac{\delta}{\delta x^i}$ and $\frac{\pa}{\pa y^i}$ as \emph{horizontal} and \emph{vertical } (partial) derivatives, respectively.

 Let $M$ be a smooth manifold, the restriction of a Finsler structure $F$ to any specific tangent space $T_xM$ gives what is known as a \emph{Minkowski norm} on $T_xM$.
  A Finsler manifold $(M,F)$ is called \emph{locally Minkowskian} if, at every point $x \in M$, there is a local chart $(x^i,U)$, with the induced tangent space coordinates $y^i$, such that the Finsler structure $F(x,y)$ has no dependence on the $x$.

 Consider the  quantity
$B_{jkl}^i: = \frac{{{\partial ^3}{G^i}}}{{\partial {y^j}\partial {y^k}\partial {y^l}}}
$
 and the well-defined tensor field on $TM_0$ given by $\textbf{B}: = B_{jkl}^id{x^j} \otimes d{x^k} \otimes d{x^l} \otimes {\partial _i}$,  called the \emph{Berwald curvature}. It is well known that $F$ is a {\it Berwald metric} if and only if  $\textbf{B}$ vanishes.
\newpage
The \emph{Berwald connection} is a torsion-free, but not necessarily metric-compatible connection.
We denote the \emph{horizontal} and  the \emph{vertical} Berwald covariant derivatives of a vector field  by $D_k$ and $\dot D_k$ and in a local coordinate system by;
\begin{eqnarray}\label{Def;BerwaldConn}
& {D_k}{X^i} = \frac{{\delta {X^i}}}{{\delta {x^k}}} + {X^r}G_{rk}^i,\\
& {{\dot D}_k}{X^i} = \frac{{\partial {X^i}}}{{\partial {y^k}}} + {X^r}C_{rk}^i,\nn
\end{eqnarray}
 where $G_{rk}^i:=\frac{{\partial G_r^i}}{{\partial {y^k}}}$.
 The connection $1-$form of Berwald connection is given by
$
{}^b\omega^i_j ={\Gamma}_{jk}^{i}dx^j+\dot A^i_{~jk} dx^k,
$
  wherein,
$\Gamma^{i}_{jk}:=\frac{1}{2}g^{il}(\frac{\delta g_{lk}}{\delta x^{j}}+\frac{\delta g_{jl}}{\delta x^{k}}-\frac{\delta g_{jk}}{\delta x^{l}})$ and  $\dot A^i_{jk }:= A^i_{jk|s} \ell ^s$.

  The \emph{Cartan connection} is a metric-compatible and h-torsion free connection
   where, the  \emph{connection $1-$form} is given  by
$
{}^c\omega^l_i :={\Gamma}_{ij}^{l}dx^j+C^l_{~ij}\delta y^j.
$

The \textit{horizontal} and \textit{vertical}  \emph{Cartan} covariant derivatives of an arbitrary $(1,1)$-tensor field $S$ are denoted by  $\nabla_kS$  and  $\dot\nabla_kS$
 and in a  local coordinate system, we have
\begin{eqnarray}\label{Def;hvCartanCon}\label{Def;CartanConn}
& \nabla_{k} S^{j}_{i}=\delta_{k}S_{i}^{j}-S_{r}^{j}\Gamma ^{r}_{ik}+S_{i}^{r}\Gamma ^{j}_{rk},\\
& \dot{\nabla }_{k} S^{j}_{i}=\dot{\partial }_{k}S_{i}^{j}-S_{r}^{j}C ^{r}_{ik}+S_{i}^{r}C ^{j}_{rk},\nn
\end{eqnarray}
respectively, where $ \dot{\partial }_{k}=\frac{\pa}{\pa y^k}$ and $ \delta _{i}:=\frac{\delta}{\delta x^{i}} $.
In a local coordinate system, the components of \emph{Cartan} \emph{hh-curvature tensor} are given by
\begin{equation*} \label{77}
R^{i}_{jkm}=\delta_{k}\Gamma^{i}_{jm}-\delta_{m}\Gamma^{i}_{jk}+
\Gamma^{i}_{s k}\Gamma^{s}_{jm}-\Gamma^{i}_{s m}\Gamma^{s}_{jk}
+R^{s}_{km}C^{i}_{s j},
\end{equation*}
where,
$R^{i}_{km}:=y^{p} R^{i}_{ pkm}$.
 In a local coordinate system the \emph{Cartan} hv-\emph{curvature tensor} and  the \emph{Cartan} vv-\emph{curvature tensor} are given by
 \begin{align}
 &P_{jkl}^i = \nabla^{i} C_{kjl}-\nabla_{j} C^{i}_{kl}+ C^{i}_{kr}\nabla_{0} C^{r}_{jl}- C^{r}_{kj}\nabla_{0} C^{i}_{rl},\label{Def,P}\\
& Q^i_{jkl} = C^i_{lr}C_{jk}^r - C_{rk}^iC_{jl}^r,\label{Def,Q}
 \end{align}
 respectively, where $\nabla_{0}:=y^i\nabla_{i}$.
 We have also the following Ricci identity for Cartan connection.
  \begin{align}\label{Eq;RicciIdentity}
   \nabla_{k} \nabla_{l} \Psi^{i} -\nabla_{l} \nabla_{k} \Psi^{i} =\Psi^{r} R^{i}_{rkl}-\dot{\nabla }_{r} \Psi^{i} R^{r}_{0kl}-\nabla_{r} \Psi^{i} S^{r}_{kl},
  \end{align}
where $S_{kl}^r = C_{kl}^r - C_{lk}^r$, see \cite[p. 19, Eq. (9.4)]{AH}.
 The components of the \emph{hh-curvature tensor } of  Berwald connection are defined by
\begin{align}\label{mv}
K_{jkl}^i = (\delta_{k} G^{i}_{jl} - G_{jls}^iG_k^s) - (\delta_{l} G^{i}_{jk}- G_{jks}^iG_l^s) + G_{rk}^iG_{jl}^r - G_{lr}^iG_{jk}^r,
\end{align}
 where $G_{jks}^i:=\frac{{\partial G_{jk}^i}}{{\partial {y^s}}}$, see \cite[p. 205]{AH}. Contracting $i=k$, we have  ${K_{jl}}: = K_{jil}^i$.
 For every $y\in T_xM$,  let us denote the trace of the Cartan hh-curvature tensor by 
\begin{align}\label{Rik}
{\bf R}_y (u)& := R^i_{\ k} (y) u^k \; \pxi,
\end{align}
where the  components are
\begin{align}
 R^i_{\  k} (y)&:= 2 {\pa G^i\over \pa x^k}
- {\pa^2 G^i\over \pa x^j \pa y^k} y^j + 2 G^j {\pa^2 G^i\over \pa y^j \pa y^k} - {\pa G^i\over \pa y^j} {\pa G^j \over \pa y^k}.\nn
\end{align}
Given any arbitrary plane $P=span\{y,v\}\subset T_xM$, the {\it flag curvature} of the plane $P$ with the flagpole $y$  is defined by
\begin{align}\label{Def;flag}
{ {\bf K}}(p, y) := { \textbf{g}_y({\bf R}_y (v), v) \over \textbf{g}_y(y,y) \textbf{g}_y(v,v)-\textbf{g}_y(v,y)\textbf{g}_y(v, y) }.
\end{align}
A Finsler structure $F$ is said to be of \emph{scalar flag curvature} or \emph{isotropic} $\textbf{K}=\textbf{K}(y)$, if for any non-zero tangent vector $y\in T_xM$ and any flag $P$ containing $y$,
${\bf K}(p, y)$ is independent of the plane $P$. In this case we have
\[
{\bf R}_y = {\bf K}(y)\Big\{\textbf{g}_y(y,y)\textsc{I}_y - \textbf{g}_y(y, .)y\Big\},\  y\in T_xM,\ x\in M,
\]
where $\textsc{I}_y:T_xM \longrightarrow T_xM$ denotes the identity mapping and $\textbf{g}_y(y, .) = \frac{1}{2} [F^2]_{y^i}dx^i$.

\begin{lem}\label{beh16}
Let $(M,F)$ be a Finsler manifold. For Berwlad connection we have the following relation
\begin{align*}
&{D_m}{D_k}\Psi  - {D_k}{D_m}\Psi=K_{0mk}^r{\dot \partial _r}\Psi,
\end{align*}
where $\Psi$ is  a $1$-homogenouse function and $K $ is a $hh$-curvature in Berwald connection.
\end{lem}
\begin{proof}
Using twice the formula for Berwald covariant derivative \eqref{Def;BerwaldConn} we have
\begin{align*}
{D_m}{D_k}\Psi  - {D_k}{D_m}\Psi & = ({\delta _m}{\delta _k}\Psi  - {\delta _r}\Psi G_{mk}^r) - ({\delta _k}{\delta _m}\Psi  - {\delta _r}\Psi G_{km}^r)\\& = {\delta _m}{\delta _k}\Psi  - {\delta _k}{\delta _m}\Psi  = ({\delta _m}{\delta _k} - {\delta _k}{\delta _m})\Psi  = \left[ {{\delta _m},{\delta _k}} \right]\Psi.
\end{align*}
Using the brackets  $\left[ {{\delta _m},{\delta _k}} \right]=K_{0mk}^r{\dot \partial _r},$  where $K_{0mk}^r:=y^jK_{jmk}^r$ are the components of  Berwald $hh$-curvature tensor, see for instance \cite{BCS}.
\begin{align*}
&{D_m}{D_k}\Psi  - {D_k}{D_m}\Psi=K_{0mk}^r{\dot \partial _r}\Psi.
\end{align*}
Therefore we have the proof.
\end{proof}
  We need  the following theorem in the sequel;
\begin{theo}\label{Th;Deicke}\cite{BCS}(Deicke)
Let $F$ be a Minkowski norm on $T_xM$. The following three statements are equivalent:\\
 (a) $F$ is Euclidean. That is, it arises from an inner product.\\
 (b) $C_{ijk} = 0$ for all $i, j, k$.\\
  (c) $C^k := g^{ij} C_{ijk} = 0$ for all $k$.
\end{theo}
\subsection{Non-Riemannian quantities and special Finsler spaces}
Using the  volume form on $\mathbb{R}^n $,  one can define  the distortion scalar function $\tau (x, y)$  on $T{M_0}$  as follows, see \cite{Sh2}.
\[ \tau (x, y):=\ln \Big [ { \sqrt{\det \Big ( g_{ij}(x,y) \Big )}
\over {\rm Vol} ({\rm B}^n(1)) }
 {\rm Vol} \Big \{ (y^i) \in \mathbb{R}^n \Big | \ F\Big (y^i \pxi|_x \Big )
  < 1 \Big \} \Big ].
  \]
Consider the {\it mean Cartan torsion} defined by ${\bf I}_y := I_i(x,y)dx^i$ where
\[
I_i(x,y): = {\pa \tau \over \pa y^j}(x,y)=
 {1\over 2} g^{jk}(x,y) {\pa g_{jk}\over \pa y^i}(x,y).
 \]
  According to Deicke's Theorem \ref{Th;Deicke},  $F_x$ is Euclidean
at $x\in M$ if and only if ${\bf I}_y=0$, or equivalently,
$\tau=\tau(x)$  at $x\in M$, see \cite{BCS}.
%

Consider the tensor field  $\textbf{L}$ with the components $L_{ijk}:= \nabla_0C_{ijk}$, called the \textit{Landsberg tensor}, and recall that $\nabla_{0} =D_{0}$ in this case.
 Let us denote by $J_i:=g^{jk}L_{ijk}$  the components of the tensor field {\bf{J}} called  the {\it {mean Landsberg tensor}}.

 A Finsler metric is called a {\textit{Landsberg metric}} (resp. {\it {weakly Landsberg metric}}) if $\bf L=0$  (resp. $\bf J=0$).

As an example of an isotropic mean Landsberg metric, we have the generalized  Funk metric on the unit ball ${\rm B}^n \subset \R^n$ satisfying ${\bf J}+ c F {\bf I}=0$ for some constant $c
\not=0$, see \cite{ChSh,Sh3}.
 A \emph{general relatively isotropic mean Landsberg} metric is a  more general Finsler structure, defined by $${J_i} + \lambda(x,y) {I_i} = 0.$$
 We can easily prove the following theorem with a technique of Akbar-Zadeh used by one of the present authors in \cite{B2}.
\begin{thm}
Let $(M,F)$ be a compact and general relatively isotropic mean Landsberg space, that is, ${\bf J}=\lambda F{\bf  I}$, for some constant $\lambda$. Then,
$(M,F)$ is either  Riemannian or  weakly Landsbergian space.
\end{thm}
It is a natural question  to study the rate of change of distortion along the geodesics $\sigma (t)$.
\begin{defn}\label{Def;S}
The \ {S-curvature} ${\textbf S }={ \textbf S}(x,y)$ is defined by
\begin{align}\label{ab}
{ \textbf S}(x,y): = \frac{d}{{dt}}\left[ {\tau (\sigma (t),\dot \sigma (t))} \right]\left|_{t = 0}. \right.
\end{align}
\end{defn}
It is positively $\textbf{y}$-homogeneous of degree one, that is,
$
{\bf S}(x,\lambda y) = \lambda {\bf S}(x,y),$ $ \lambda  > 0.$
  For a geodesic $\sigma (t)$ define
$\tau (t): = \tau (\sigma (t),\dot \sigma (t)),$ and $\bf S(t): =\bf S(\sigma (t),\dot \sigma (t))$.
By means of \eqref{ab},  $\bf S(t) = \tau '(t)$ and  if $\bf S=0$, then $\tau (t) = constant$.
Intuitively, the distortion of an infinitesimal color pattern in the distortion $\dot \sigma (t)$ does not change along the geodesic $\sigma  = \sigma (t)$. However, the distortion might take different values along different geodesics. In the case when $F$ is a Berwald metric, the infinitesimal  color patterns do not change over the ambient manifold \cite{S},
  therefore the distortion of the pattern in the direction $\dot \sigma (t)$ does not change along any geodesic $\sigma=\sigma(t)$.
A Finsler structure $F$ is said to have \emph{isotropic S-curvature} if
$
{\bf S} = (n + 1)cF,
$
where $c=c(x)$ is a scalar function on $M$.
Differentiating the S-curvature twice, gives rise to the following quantity.
\begin{align}
{E_{ij}}: = \frac{1}{2}{{\bf S}_{{y^i}{y^j}}}(x,y): = \frac{1}{2}{{\bf S}_{{.i}{.j}}}(x,y).
\end{align}
 For all  $y \in {T_x}M\backslash 0,\;\;{\textbf{E}_y} = {E_{ij}}(x,y)d{x^i} \otimes d{x^j}$ is a symmetric bilinear form on ${T_x}M$ \cite{S}. The Finsler structure $F$ is called a \emph{weak Berwald metric} if $\textbf{E}=0$.\\
Let us consider the  non-Riemannian quantity $\textbf{H}_y:=\textbf{H}_{ij}(y)dx^i\otimes dx^j $,  with the components  $\textbf{H}_{ij}:= \nabla_{0} E_{ij}= y^k\nabla_{k} E_{ij}$.
%
%
%
%
\subsection{The two Ricci tensors of Akbar-Zadeh in Finsler geometry}
There are two definitions of the Ricci tensor in Finsler geometry, introduced by Akbar-Zadeh, which are used to define Einstein spaces. After a brief review of the methods applied for these definitions, we will study their relationship and show that the Lie derivatives of these two Ricci tensors are equal along the $I$-projective vector fields.

Recall that we  denote the hh-curvature of Berwald connection and its contraction  by  $ K^m_{jkl}$  and ${K_{jl}}: = K_{jml}^m$, respectively. Akbar-Zadeh has considered first the following symmetric tensor as \emph{Ricci tensor} in Finsler geometry.
\be\label{Eq;Ricci1}
{ \tilde R_{ij}}:=\frac{1}{2}({K_{ij}} + {K_{ji}}).
\ee
Next he has introduced the \emph{second Ricci tensor} using the vertical partial derivative of  the  trace of the Cartan hh-curvature, by
\be\label{Eq;Ricci2}
 Ric_{ij} := ( \frac{1}{2} F^2 {\cal R}ic)_{y^iy^j}=\frac{1}{2}(y^j R_{jml}^my^l)_{y^iy^j},
 \ee
 where   ${\cal R}ic(x,y):= R^i_i$.
 Finally,  he considered the collection of all Finsler metrics $g$ defined on $M$ and used the Einstein-Hilbert functional method in general relativity. Following Hamilton's work in Riemannian geometry,  Akbar-Zadeh introduced the definition of Einstein-Finsler metrics as critical points of the Einstein-Hilbert functional.
More intuitively, he considered the scalar function $\hat{H}=\tilde{H}-\lambda(x)\mathcal{R}ic$ on $S(M)$, where $\tilde{H}=g^{jk}Ric_{jk}$ and $\lambda$ is a differentiable function on $M$.
Akbar-Zadeh considers the energy-functional  $$E(g)=\underset{{S(M)}}\int\hat{H} d\mu,$$ and find the metric $g$ as a solution for $E'(g)=0$.
The obtained definition in this method coincides with the second definition of Ricci  tensor mentioned above.  One of the advantages of the second definition of Ricci tensor  is its independence to the choice of the Cartan, Berwald or Chern(Rund) connections.

 The second Ricci tensor  is used  to define the Ricci flow and show the existence and uniqueness of its solution  on  Finsler surfaces, see \cite{BiSe}. The second definition has  similar geometrical applications as the Ricci scalar and, by simple calculations, we have
 \begin{align}\label{Eq;Ric}
   {\cal R}ic &=\ell^i\ell^kRic_{ik}.
 \end{align}
Taking into account the first and the second Ricci tensors, one can consider the following two definitions of  \emph{Finsler-Einstein} metrics.
\begin{align}
  \tilde R_{ij}& = c(x)g_{ij},\label{Eq;Ric1}\\
  { R}ic_{ij}& = (n-1)k(x)g_{ij},\label{Eq;Ric2}
\end{align}
respectively, where $c(x)$ and $k(x)$ are functions on $M$.

Using \eqref{Eq;Ric} one can see that ${\cal R}ic = (n - 1)k(x)$ and hence the definition \eqref{Eq;Ric2} is equivalent to
\begin{align}
 { R}ic_{ij} ={\cal R}ic\  g_{ij}.
\end{align}
Note that in this work, we use the Finsler-Einstein metrics corresponding to the first definition of  Ricci tensor.
%

\begin{lem}\label{Le;AZRiccis}
Let $(M,F)$ be a Finsler manifold.
The two  Ricci curvatures of Akbar-Zadeh, namely
 $Ric_{ij}$ and $\tilde R_{ij}$ are related by
\begin{align}\label{AZRiccitensors}
Ri{c_{ij}} = {\tilde R_{ij}} - {\textbf{H}_{ij}}.
\end{align}
\end{lem}
\begin{proof}
Recall that we have
$Ric = R_m^m = {K_{00}} = {K_{ij}}{y^i}{y^j}$, and ${K_{jl}}: = K_{jml}^m$, where $K_{jkl}^i$, is the hh-curvature of Berwald connection.
Consider the  non-Riemannian quantity defined earlier by $\textbf{H}_{ij}= \nabla_{0} E_{ij}$. Contracting the Bianchi identity in Berwald connection,
 and taking into account the equation $ G^i_{jil}=2E_{jl}$ yields the following equations, see \cite[p.4]{Ra2} for more details.
\be\label{Eq;Kij}
 y^jK_{jl}.m=0, \qquad y^lK_{jl}.m=-2\textbf{H}_{jm}.
 \ee
Hence,
\begin{align*}
Ri{c_{ij}} &= \frac{1}{2}\frac{{{\partial ^2}Ric}}{{\partial {y^i}\partial {y^j}}} = \frac{1}{2}\frac{{{\partial ^2}}}{{\partial {y^i}\partial {y^j}}}{K_{00}}\\& = \frac{1}{2}\frac{\partial }{{\partial {y^i}}}(\frac{\partial }{{\partial {y^j}}}({K_{kl}}{y^k}{y^l}) = \frac{1}{2}{{\dot \partial}_i}\{ {K_{kl.j}}{y^k}{y^l} + {K_{jl}}{y^l} + {K_{kj}}{y^k}\} \\ &= \frac{1}{2}\{ {K_{jl.i}}{y^l} + {K_{ji}} + {K_{kj.l}}{y^k} + {K_{ij}}\} \\& =\frac{1}{2}({K_{ij}} + {K_{ji}}) - { \bf H_{ij}}={ \tilde R_{ij}}  - {\bf H_{ij}}.
\end{align*}
Therefore we have the proof.
\end{proof}


\subsection{Volume form and  Green's Theorem in Finsler geometry}
If we denote by $p:TM_0\longrightarrow M$ the fiber
bundle of non-zero tangent vectors of $M$ then for the fixed point $x = {x_0} \in M$, the fiber manifold ${V_n} = {p^{ - 1}}({x_0})$ will be
equipped with the Riemannian metric $g(X, Y )$, where $X$ and $Y$ are tangent vector fields on
$V_n$ in $z \in {p^{ - 1}}({x_0})$. The hypersurface $S_{x_0}$ of $V_n$ satisfying $F(x,v) = 1$ is called the \emph{indicatrix} in ${x_0} \in M$.

Let the \emph{sphere bundle} $SM$ be the bundle of all directions or rays on $T M_0$.
The set $SM =\{(x,y) \in TM_0 | F(x,y) = 1\}$ is  called the \emph{unit sphere bundle} or \emph{indicatrix bundle}.
 A moment's thought shows that it is diffeomorphic to the {indicatrix bundle}, which is a $(2n-1)$ dimensional subbundle of $T M_0$.

At every point  $z\in SM$  we denote a  \emph{horizontal vector field}  $X:M \to SM$, by $X:=X^i(z)\frac{\delta}{\delta x^i}$,  and the  corresponding \emph{horizontal1-form} on $SM$ again by $  X:= {{X_i}(z)} d{x^i}$ , if no ambiguity appears.

Let us consider at every point  $z\in SM$  a unitary horizontal vector field and  the corresponding horizontal 1-form  on $SM$ by   $U=U^i(z)\frac{\delta}{\delta x^i}$, and $\omega  = {{U_i}(z)} d{x^i}$, respectively.
We denote by ${(d\omega )^{n - 1}}$ the $(n-1)$-th exterior power of $d\omega$.
 A \emph{volume
element} used by Akbar-Zadeh on the fiber bundle $SM$  is a $(2n-1)$-form on $SM$ defined by;
\begin{align}\label{Def;Vol}
\eta  = \frac{{{{( - 1)}^{\frac{{n(n - 1)}}{2}}}}}{{(n - 1)!}}\omega  \wedge {(d\omega )^{n - 1}}.
\end{align}
The sphere bundle $SM$ is compact if the manifold $M$ is compact.
As $SM$ is always orientable, we do not need to assume $M$ is orientable to integrate on $SM$, in Finsler geometry.

  For any given tensor fields $W$ and $T$ on $SM$,  there is a \emph{global scalar product} defined on $SM$, by
$
\left\langle{W,T}\right\rangle:= \underset{SM}\int{{W^{{r_1}{r_2}...{r_k}}}} {T_{{r_1}{r_2}...{r_k}}}\eta,
$
where ${W^{{r_1}{r_2}...{r_k}}} = {g^{{r_1}{s_1}}}{g^{{r_2}{s_2}}}...{g^{{r_k}{s_k}}}{W_{{r_1}{r_2}...{r_k}}}$.

The {\it co-differential operator}
$\delta$ on the differential forms defined on $S(M)$, is adjoint to the {\it differential operator} $d$, in the global scalar product defined on $S(M)$.
The\emph{ horizontal} and \emph{vertical differential operators } and   \emph{co-differential operators} of a horizontal 1-form $X = {X_i}(z)d{x^i}$ are defined by:
\begin{align}
&dx = \frac{1}{2}({D_i}{X_j} - {D_j}{X_i})\;\;d{x^i} \wedge d{x^j},\nn\\
&\dot dx = ( - \frac{{\partial {X_i}}}{{\partial {y^j}}}d{x^i} \wedge d{y^j}),\nn\\
&\delta X = - ({\nabla ^j}{X_j} - {X_j}{\nabla _0}{T^j}) = - {g^{ij}}{D_i}{X_j},\label{div.operators+1}\\
&\dot \delta X = - F{g^{ij}}\frac{{\partial {X_i}}}{{\partial {y^j}}},\label{div.operators+2}
\end{align}
respectively, where $ \nabla _i$ is the Cartan h-covariant derivative,  $\nabla _0:=y^i\nabla_i$ and $D_i$ is the Berwald h-covariant derivative.  Here the co-differential operator $\delta$ and $\dot \delta$ are the formal adjoint to $d$ and $\dot d$ respectively,  in the above global scalar product over $SM$.

Let $(M,F)$ be a Finsler manifold with the volume element $\eta$.
The\emph{ divergence operator } $div : T(M) \mapsto C^\infty(SM)$ is defined by
$d(i_{_X}\eta) = (div X)\eta,$
where $i_{_X}$ denotes the \emph{interior product} by $X$, and for any $k$-form $\omega$, $i_{_X}\omega$ is the $(k -1)$-form defined by
$i_{_X}\omega(V_1,...,V_{(k-1)})=\omega(X,V_1,...,V_{(k-1)}).$
One can check that the divergence operator satisfies the following product rule for a smooth function $u \in C^\infty(SM),\ $
$div(uX)= u\  divX + <grad\ u,X>$.
In this paper, we say that a scalar function is a \emph{divergence term}, if it results from a contraction of the indices in all of its forms. 
In the sequel, we will use the following corollary of Green's Theorem to omit the divergence terms by integration.
 For instance, recall that the horizontal divergence  of a vector field $X$  is given by
 $div X=(\nabla_iX^i-X^i\nabla_{_0}C_i)$, where $\nabla_{_0}=y^j\nabla_{j}$,
 and we have the following theorem.
\begin{coro}\label{Cor;Green}\cite{A1}{(Green's Theorem)}
 Let $(M,F)$ be a compact Finsler manifold without boundary.  If
$X $ is a section of $SM$ then the integral of its divergence vanishes, that is,
\[
\int_{SM}(div X)\eta=0, 
\]
where $\eta$ is the volume form on $SM$ given by \eqref{Def;Vol}.
\end{coro}
 Let $(M,F)$ be a compact Finsler manifold without boundary.  The above corollary,  known as \emph{divergence formulas} for a horizontal 1-form,  are defined by
\begin{align*}
&\int\limits_{SM} ( \delta X)\eta = - \int\limits_{SM} {({g^{ij}}} {D_i}{X_j})\eta = 0,\ \textrm{and}
&\int\limits_{SM} ( \dot \delta X)\eta = - \int\limits_{SM} {F{g^{ij}}} \frac{{\partial {X_i}}}{{\partial {y^j}}}\eta = 0,
\end{align*}
 where the co-differential operators $\delta$ and $\dot \delta$ are defined by \eqref{div.operators+1} and \eqref{div.operators+2}   respectively.

 Let ${\pi _1} = {a_i}(z)d{x^i}$, be a horizontal 1-form on $S(M)$, it is proved in \cite{AH} that
\begin{align}\label{bz}
\delta {\pi _1} =  - ({\nabla ^j}{a_j} - {a_j}{\nabla _0}{C^j}),
\end{align}
where $\nabla ^j:=g^{ij}\nabla _i$ is the Cartan h-covariant derivative,  $\nabla _0:=y^i\nabla _i$ and  the vector field ${C^j:=g^{ik}C^j_{ik}}$ is the  trace of torsion tensor.
 Similarly, let ${\pi _2} = {b_j}\delta {y^j}$
  be a vertical 1-form on $SM$, we have \cite{AH}
\begin{align}\label{bw}
\delta {\pi _2} = - F{g^{ij}}\frac{\partial {b_i} }{{\partial{y^j}}}.
\end{align}
Using the volume form, the``average Ricci curvature" can be defined as follows.
\begin{defn}\label{Def;average Ricci}
 The \bf{average Ricci curvature}  $\tilde{R} (X,X)$  is defined by
\begin{align*}
\tilde{R}(X,X)=\frac{\int\limits_{SM} R_{ij}X^{i}X^{j}\eta(g)}{\underset{SM}\int \eta(g)},
\end{align*}
where $$\int\limits_{SM}\eta(g):=Vol(SM),$$ is the volume of $SM$.
\end{defn}

\subsection{Projective vector fields on Finsler spaces}
Every vector field $X$ on a differentiable manifold $M$ induces naturally an infinitesimal coordinate transformations
$(x^i,y^i)\longrightarrow(\bar{x}^i,\bar{y}^i)$ on $TM$, given by
$\bar{x}^i=x^i+X^idt,$ and $\bar{y}^i=y^i+y^k\frac{\partial X^i}{\partial x^k}dt.$
It leads to the notion of the \textit{ complete lift} $\hat{X}$ of a vector field $X$ on $M$ to a vector field $\hat{X}=X^i\frac{\partial}{\partial x^i}+y^k\frac{\partial X^i}{\partial x^k}\frac{\partial}{\partial y^i}$ on $TM_0$, see for instance \cite{Yano}.

Let $X$ be a vector field on the Finsler manifold $(M,F)$. We denote its { complete lift} to $TM_0$ with respect to the horizontal and vertical   basis frames$\frac{\delta}{\delta x^i}$ and $\frac{\pa}{\pa y^i} $  by
$\hat{X}$ where, $\hat{X}=X^i+y^j\nabla_jX^i\frac{\pa}{\pa y^i}$ and $\nabla^i$ is the h-covariant derivative of Cartan connection..
 It's a remarkable observation that, $\pounds_{\hat{X}}y^i=0$, $\pounds_{\hat{X}}dx^i=0$ and the differential operators $\pounds_{\hat{X}}$,
$\frac{\partial}{\partial x^i}$, the exterior differential operator $d$ and $\frac{\partial}{\partial y^i}$ commute,  see for instance \cite{AH,Yano}.

In Finsler geometry, almost all geometric objects depend on both position and direction. Hence, the Lie derivatives of these objects in direction of a vector field $X$ on $M$ must be considered concerning to the complete lift vector field $\hat{X}$.

A diffeomorphism between the two Finsler manifolds $(M,F)$ and $(M,\bar F)$  is called a \emph{projective transformation} if it takes every forward (resp. backward) geodesic  to a forward (resp. backward) geodesic. A projective transformation is called an \emph{affine transformation} if it leaves invariant the connection coefficients.

A smooth vector field $X$  is called a \emph{projective vector field} or \emph{affine vector field} on $(M,F)$ if  the associated local flow  is a projective or affine transformation, respectively. There are several approaches for the definition of a projective vector field on a Finsler manifold. We frequently use the following Lemma.
\begin{lemA}\label{Def;proj.vect}\cite{Ti}
A vector field $X$ on the Finsler manifold $(M,F)$ is  a \emph{projective vector field}  if and only if there is a function $\Psi = \Psi(x, y)$ on $TM_0$, positively  $1$-homogeneous on $y$, such that
\begin{align}\label{proj 3}
{\pounds_{\hat X}}{G^i} = \Psi(x, y) {y^i}.
\end{align}
$X$ is an \emph{affine vector field} if and only if $\Psi(x, y)= 0$.
\end{lemA}
A projective vector field can be also characterized by the following Lemmas.
\begin{lemA}\label{prop2}\cite{A2}
$X$ is a projective vector field on the Finsler manifold $(M,F)$, if
and only if
\begin{equation}
\nabla_{0} ({\pounds}_{\hat{X}}g_{ij})=2\Psi g_{ij}+\Psi_iy_j+\Psi_jy_i,\label{proj char 1}
\end{equation}
where $\nabla_{0} ({\pounds}_{\hat{X}}g_{ij}):= y^r\nabla_{r} ({\pounds}_{\hat{X}}g_{ij})$.
\end{lemA}
\begin{lemA}\label{la}\cite{AH}
Let $(M,F)$ be a Finsler manifold. A projective vector field satisfies;
\begin{equation}\label{md}
{\cal{L}}_{\hat{X}}\nabla_{0}C^{i}_{jk}=\nabla_{0}{\cal{L}}_{\hat{X}}C^{i}_{jk}+\Psi C^{i}_{jk}+y^{i}C^{s}_{jk}\Psi_{s},
\end{equation}
where
\begin{equation}\label{mf}
{\cal{L}}_{\hat{X}}C^{i}_{jk}=\dot{\nabla }_{k} (\nabla_{j} X^{i}+C^{i}_{jh}\nabla_{0} X^{h})+X^{h}P^{i}_{\ jhk}+\nabla_{0} X^{h}Q^{i}_{\ jhk}.
\end{equation}
\end{lemA}
\section{Some formulas on projective vector fields }
Here, we explicitly derive some formulas on the Lie derivatives of some significant geometric objects along projective vector fields.
\begin{lem}\label{Lem;proj}
Let $(M,F)$ be a Finsler manifold. For a
 projective vector field $X$ on $(M,F)$ we have
\begin{align}
1.\ &{\cal{L}}_{\hat{X}}{G}^i_k ={\Psi _k}{y^i} + \Psi \delta _k^i \qquad \textrm{where} \quad \Psi_k:=\Psi_{.k}:=\frac{{\partial \Psi }}{{\partial {y^k}}}.\label{lem;a}\\
2.\ &\pounds_{\hat{X}}G^i_{ jk}=\delta^i_{\ j}\Psi_k+\delta^i_{k}\Psi_j+y^i\Psi_{k.j}.\label{bm}\\
3.\ &\pounds_{\hat{X}}G^i_{jkl}=\delta^i_{j}\Psi_{k.l}+\delta^i_{k}\Psi_{j.l}
+\delta^i_{l}\Psi_{k.j}+y^i\Psi_{k.j.l}. \label{proj 400} \\
4.\ &\pounds_{\hat{X}}E_{jl}=\frac{1}{2}(n+1)\Psi_{j.l}. \label{proj 6}\\
5.\ &\pounds_{\hat{X}}I_k=f_{.k}, \  \textrm{where} \ f=\nabla_{i} X^{i}+I_{i}\nabla_{0} X^{i}. \label{proj 7}\\
6.\ &\pounds_{\hat{X}}J_k=\nabla_{0} f_{.k}+\Psi I_k. \label{proj 8}\\
7.\ &(n+1)\Psi_k=\nabla_{k} f+\nabla_{0} f_{.k}, \textrm{where} \; \Psi=(\frac{1}{n+1})\nabla_{0} f.\label{lem;g}\\
8.\ & \pounds_{\hat{X}}K^i_{jkl}=\delta^i_{j}(D_{k}\Psi_{l} -D_{l}\Psi_{k})
+\delta^i_{l}D_{k}\Psi_{j} \nn \\ & \qquad \qquad -\delta^i_{k}D_{l}\Psi_{j} +y^i(D_{k}\Psi_{l} -D_{l}\Psi_{k})_{.j}\label{proj 1000}\\
9.&\pounds_{\hat{X}}K_{jl}=D_{j}\Psi_{l} -nD_{l}\Psi_{j} +y^i{D_{i}} \Psi_{l.j} .\label{proj 11}
\end{align}
\end{lem}
\begin{proof}
Let $(M,F)$ be a non-Riemannian Finsler manifold and $X$ a projective vector field  on $(M,F)$.
After  Lemma \ref{Def;proj.vect}, the characteristic function $\Psi(x,y)$ on $TM_0$, is positively  $1$-homogeneous on $y$.
By a vertical derivative of \eqref{proj 3} we have the first assertion \eqref{lem;a}.
Again, a vertical derivative of \eqref{lem;a} leads to  the second assertion \eqref{bm}.
Another   vertical derivative of \eqref{bm} yields the third assertion \eqref{proj 400}.
By  definition of the {\textbf{E}-curvature} we have ${E_{ij}} := \frac{1}{2}G_{irj}^r$, hence ${\pounds}_{\hat{X}}{E_{ij}} = \frac{1}{2}{\pounds}_{\hat{X}}G_{irj}^r$.
 Contracting the indices $i$ and $k$ in \eqref{proj 400} and using Euler's Theorem, we get
\begin{align*}
2{\pounds}_{\hat{X}}{E_{jl}} = {\pounds}_{\hat{X}}G_{jil}^i = {\Psi _{i.j.l}}{y^i} + {\Psi _{i.j}}\delta _l^i + {\Psi _{i.l}}\delta _j^i + {\Psi _{j.l}}\delta _i^i,
\end{align*}
or
$2{\pounds}_{\hat{X}}{E_{jl}} = {\Psi _{l.j}} + (n + 1){\Psi _{j.l}}$ which proves the $4^{th}$ statement of Lemma.
 Contracting $i$ and $k$ in \eqref{mf}, yields
\begin{align}\label{bk}
{\pounds}_{\hat{X}}{I_k} =(\nabla_{i} X^{i}+I_{h}\nabla_{0} X^{h})_{.k} + {X^h}P_{ihk}^i + \nabla_{0} X^{h}Q_{ihk}^i.
\end{align}
Using \eqref{Def,P} and \eqref{Def,Q} the second and the third curvature tensors of Cartan connection respectively, we see that the second and the third terms in the right hand side of \eqref{bk} vanish and therefore
${\pounds}_{\hat{X}}{I_k} = (\nabla_{i} X^{i}+I_{h}\nabla_{0}X^{h})_{.k}.$
If we set  $f :=  (\nabla_{i} X^{i}+I_{h}\nabla_{0} X^{h})$   then ${\pounds}_{\hat{X}}{I_k} = {f_{.k}},$ which is the equation \eqref{proj 7} in the statements of Lemma \ref{Lem;proj}.
Contracting $i$ and $k$ in \eqref{md} gives
\begin{equation}\label{cf}
{\pounds}_{\hat{X}} \nabla_{0} I_{k}=  \nabla_{0} ({\pounds}_{\hat{X}} {I_k})+ \Psi {I_k}.
\end{equation}
Using the definition ${J_k} := \nabla_{0} I_{k}$,  \eqref{proj 7} and \eqref{cf} we obtain the statement $\eqref{proj 8}$  of Lemma \ref{Lem;proj}.
A simple manipulation leads   $$\Psi=(\frac{1}{n+1})\nabla_{0}(\nabla_{i} X^{i}+I_{i}\nabla_{0}X^{i})=
\frac{1}{n+1}\nabla_{0} f.$$
  A vertical partial derivative of the last equation yields the $\eqref{lem;g}$ statement of Lemma \ref{Lem;proj}.
 Lie derivative of the curvature tensor  \eqref{mv}, and using \eqref{bm} and \eqref{proj 400} yields
\begin{align*}
& \pounds_{\hat{X}}K^i_{jkl}=\delta^i_{j}(D_{k}\Psi_{l} -D_{l}\Psi_{k})
+\delta^i_{l}D_{k}\Psi_{j}-\delta^i_{k}D_{l}\Psi_{j} +y^i(D_{k}\Psi_{l} -D_{l}\Psi_{k})_{.j}.
\end{align*}
Hence we have the  statement  $\eqref{proj 1000}$.
 Contracting $i$ and $k$ in the last equation yields the  statement $\eqref{proj 11}$ of Lemma \ref{Lem;proj}.
This completes the proof of Lemma
\end{proof}
Recall that in Riemannian geometry, in the definition of projective vector fields, the function $\Psi=\Psi(x,y)$   is linear for $y$, while in the
Finslerian setting the mentioned linearity depends on the quantity which is a non-Riemannian feature of the space, as we saw in the $7^{th}$ statement of Lemma\ref{Lem;proj}.

 We consider in the sequel the following definitions for a vector field $X$.
 \begin{defn}\label{Def;I,E,H}
Let $(M,F)$ be a Finsler manifold and  $X$ a projective vector field on $(M,F)$. Using the notation of the last lemma we say that $X$ is;
\begin{itemize}
  \item {\bf I}-{\it invariant projective} vector field  or simply \emph{{\bf I}-projective} vector field if $\pounds_{\hat{X}}{\bf I}=0$, equivalently, $f_{.k}=0$.
  \item {\bf E}-{\it invariant projective vector field} if $\pounds_{\hat{X}}{\bf E}=0$, or
      equivalently, $\Psi_{k.j}=0$.
   \item C-{\it projective vector field} if  $\nabla_k\Psi_{j}=\nabla_j\Psi_{k}$.
   \item {\bf H}-{\it invariant projective vector field}  if
       $\nabla_l\Psi_{jk}=\nabla_k\Psi_{jl}$.
\end{itemize}
\end{defn}
In the case of $\textbf{E}$-invariant projective vector fields, the function $\Psi(x,y)$ is linear with respect to $y$ at any point $x$.
\subsection{{\bf I}-projective vector fields}
As mentioned earlier, to have consistent quantities that measure the local anisotropy, we have to assume that it is a projective invariant quantity.
 The study of the projective symmetry of ${\bf I}$, consists of considering the projective vector fields $\hat X$ which leave ${\bf I}$ invariant.

 The ${\bf I}$-projective vector fields are widely studied  in \cite{Ra1,Ra2} under the name {\it the special projective vector fields}.

    The following proposition emphasizes the essential role of symmetries.
\begin{prop}\label{Prop;I-inv}
If a Finsler manifold $(M,F)$ admits an {\bf I}-projective vector field $X$, then  $\pounds_{\hat{X}}{\bf E}=0$, $\pounds_{\hat{X}}{\bf
H}=0$,   $\pounds_{\hat{X}}{\bf B}=0$.
\end{prop}
%
%
\begin{proof}
Let $X$ be an $I$-invariant projective vector field, that is, $\pounds_{\hat{X}}{\bf I}=0$. Based on \eqref{proj 7}, we get ${f_{.k}}=0$, hence $f$  depends on $x$ alone.
By  \eqref{proj 8} we get $(n + 1){\Psi_k} = {{\nabla_kf }}$ and $\Psi  = {\Psi_k}{y^k} = {f_k}{y^k}$, where $\Psi$ is the differential of $f$.
On the basis of \eqref{proj 6}, we have $\pounds_{\hat{X}}{E_{ij}} = \frac{1}{2}(n+1){\Psi_{j.l}} = \frac{1}{2}(n + 1){f_{j.l}}=0$,
 $f$ is independent of $y$ and $\Psi$ is a linear function on $y$.
Now, let $J=0$, from \eqref{proj 8} and ${f_{.k}}$ we have $\pounds_{\hat{X}}{J_k} = \Psi {I_k} = 0$. If $\Psi  \ne 0$ then $I=0$ and if $I \ne 0$ then $\Psi  = 0$. We know that $\pounds_{\hat{X}}{\bf H_{ij}} = \frac{{(n + 1)}}{2}{\Psi_{ij}}$ since $\Psi$ is linear on $y$, the second derivative vanishes, and $\Psi_{ij}=0$, hence $\pounds_{\hat{X}}{\bf H_{ij}}=0$. Using \eqref{proj 400}, linearity of  $\Psi$ on $y$ yields $\pounds_{\hat{X}}\textbf{B }= 0$.

\end{proof}
As an application of the above proposition, we have the following Lemma.
\begin{lem}\label{beh15}
Let $(M,F)$ ba a Finsler manifold.
 Lie derivatives  of the two Ricci tensors  $Ric_{ij}$  and $\tilde R_{ij}$  coincide along the {\bf I}-projective vector fields, that is
$${{\cal{L}}_{\hat X}}Ri{c_{ij}} = {{\cal{L}}_{\hat X}}{\tilde R_{ij}}.$$
\end{lem}
\begin{proof}
Recall that the relation between the two definitions of Ricci curvature by Akbar-Zadeh, namely $Ric_{ij}$  and $\tilde R_{ij}$ defined by \eqref{Eq;Ric1} and \eqref{Eq;Ric2},  is given by
$
Ri{c_{ij}} = {\tilde R_{ij}} - {\bf H_{ij}},
$
cf. \eqref{AZRiccitensors}.
Let $X$ be an {\bf I}-projective vector field on $(M,F)$. Proposition\ref{Prop;I-inv} tells us that ${{\cal{L}}_{\hat X}}{H{ij}} = 0$. From which \eqref{AZRiccitensors} yields
$
{{\cal{L}}_{\hat X}}{Ric_{ij}} = {{\cal{L}}_{\hat X}}{\tilde R_{ij}}.
$
This completes the proof.
\end{proof}
 Lemma \ref{Lem;proj} provides proof for Theorem \ref{Th;first}.\\
  \emph{\textbf{Proof of Theorem \ref{Th;first}.}}
Let $X$ be an \textbf{I}-invariant projective vector field, hence $\Psi=y^k\nabla_k f(x)$, $\nabla_l\Psi_{k}=\nabla_k\Psi_{l}$ and $\Psi_{i.k}=0$.
By definition ${{\tilde R}_{ij}} = \frac{1}{2}({K_{ij}} + {K_{ji}})$ hence,
${{\pounds}_{\hat X}}({{\tilde R}_{ij}}) = \frac{1}{2}({{\pounds}_{\hat X}}{K_{ij}} + {{\pounds}_{\hat X}}{K_{ji}}).$
Since $\Psi_{i.j}=0$  by means of \eqref{proj 11} in the last lemma, we have
\begin{align}\label{beh}
{ {\pounds}_{\hat X}}({{\tilde R}_{ij}}) = \frac{1}{2}({D_i}{\Psi _j} - n{D_j}{\Psi _i} + {D_j}{\Psi _i} - n{D_i}{\Psi _j})  = (1 - n)({D_i}{\Psi _j}).
\end{align}
On the other hand, $(M,F)$ is an Einstein space and by definition ${ Ric_{ij}} = (n-1)kg_{ij}$, where by  assumption $k$ is a constant.
Hence the Lie derivative along $\hat X$ reads
\begin{align}\label{beh1}
{ {\pounds}_{\hat X}}({ Ric_{ij}}) = k(n-1){{\pounds}_{\hat X}}{g_{ij}}.
\end{align}
 By Lemma \eqref{beh15}
the equations \eqref{beh} and \eqref{beh1} we have
\begin{align}\label{beh2}
(1 - n)({D_i}{\Psi _j}) = k(n-1){t_{ij}},
\end{align}
where  ${t_{ij}}: = {\pounds}_{\hat{X}}{g_{ij}}$.
The contracted Berwald derivative  $D_{0}$,  is equal to the Cartan contracted derivative $ \nabla_{0}$,  and  \eqref{beh2} leads
$ (1 - n){D_0}{D_i}{\Psi _j}=k(n-1)({D_0}{t_{ij}}) $.
 Hence by \eqref{proj char 1} we get
\begin{align}\label{beh3}
(1 - n){D_0}{D_i}{\Psi _j}=k(n-1)(2\Psi g_{ij}+\Psi_iy_j+\Psi_jy_i).
\end{align}
The partial derivative  $ \frac{\pa}{\pa y^m}$  of \eqref{beh3} yields
\begin{align*}
 (1-n){D_m}{D_i}{\Psi _j}+&(1 - n){y^l}{{\dot \pa}_m}({D_l}{D_i}{\Psi _j})\\&=k(n-1)(2{\Psi _m}{g_{ij}}+4{C_{ijm}}\Psi +{\Psi _i}{g_{jm}}+{\Psi _j}{g_{im}}).
\end{align*}
As well  from the last  equation we have
\begin{align}\label{beh5}
(1 - n){D_m}{D_i}{\Psi _j} =& k(n-1)(2{\Psi _m}{g_{ij}} + 4{C_{ijm}}\Psi + {\Psi _i}{g_{jm}} + {\Psi _j}{g_{im}})  \\&- (1 - n){y^l}{{\dot \partial }_m}({D_l}{D_i}{\Psi _j}).\nn
\end{align}
We can easily see that the last term in the above equation, notably ${{\dot \partial }_i}({D_l}{D_m}{\Psi _j})$  is symmetric in the three indices, $i,m$ and $j$.
In fact, letting ${M_{ij}}:= {D_i}{\Psi _j} $, and using the horizontal Berwald derivative we have
\begin{align*}
{M_{ij.m}} = {{\dot \partial }_m}{D_i}{\Psi _j}& = {{\dot \partial }_m}(({\partial _i} - G_i^r{{\dot \partial }_r}){\Psi _j} - {\Psi _r}G_{ij}^r) \\&
= {{\dot \partial }_m}({\partial _i}{\Psi _j} - G_i^r{\Psi _{j.r}} - {\Psi _r}G_{ij}^r) =  - {\Psi_r}G_{ijm}^r.
\end{align*}
%
%
Therefore ${M_{ij.m}}$ is symmetric in $i,j,$ and $m$. Using twice \eqref{beh5}, we have
\begin{align}\label{beh7}
{D_i}{D_m}{\Psi _j} - {D_m}{D_i}{\Psi _j} = {k}({\Psi _m}{g_{ij}} - {\Psi _i}{g_{jm}}).
\end{align}


 Recalling that $\Psi$ is 1-homogenouse of degree $1$, by Euler theorem we have ${y^i}{\Psi _i} = \Psi $.
 Contracting \eqref{beh7} with $y^i$ and $y^j$ yields
\begin{align}\label{beh8}
{D_0}{D_m}\Psi  - {D_m}{D_0}\Psi  ={k}({\Psi _m}{F^2} - \Psi {y_m}).
\end{align}
Using Lemma\eqref{beh16} we get
\begin{align}
{D_0}{D_m}\Psi  - {D_m}{D_0}\Psi &= {y^k}({D_k}{D_m}\Psi  - {D_m}{D_k}\Psi ) = {y^k}(K_{0mk}^r{{\dot \partial }_r}\Psi )\nn \\ &
 =K_{0m0}^r{\Psi _r} = R_m^r{\Psi _r}.
\end{align}
Convecting the above equation with the gradient vector field ${\Psi ^m=g^{im}\Psi _i}$  and using \eqref{beh8}   yields
\begin{align*}
R_r^m{\Psi _r}{\Psi ^m} = {k}\{ {\Psi _m}{\Psi ^m}{F^2} - {\Psi _i}{y^i}{\Psi ^m}{y_m}\}.
\end{align*}
The above equation is written globally in the free index form
\begin{align}\label{3002}
  {{g_y}({\bf R}_y(\nabla {\Psi  }}),\nabla {\Psi  })={k}\{{{F^2}(y){g_y}(\nabla {\Psi }},\nabla {\Psi }) - {g_y}{{(\nabla {\Psi}},y)}^2\},
\end{align}
where $g_y(.,.)$, is the inner product induced by $F$, ${\bf R}_y$ is the Riemann curvature operator \eqref{Rik} and $\nabla {\Psi}|_{(x,y)}=g^{im}\frac{\pa \Psi}{\pa y^i} \frac{\pa}{\pa x^m}|_{(x,y)}\in \Gamma ({\pi ^*}TM)$ is the gradient vector field.

On the other hand,  there is a tangent plane $P=span\{y, \nabla \Psi\}\in T_x M$, such that the flag curvature \eqref{Def;flag} is given by
\begin{align}\label{3003}
{\bf K}(P,y) &= \frac{{{g_y}({\bf R}_y(\nabla {\Psi  }}),\nabla {\Psi  })}{{{g_y(y,y)}{g_y}(\nabla {\Psi }},\nabla {\Psi }) - {g_y}{{(\nabla {\Psi}},y)}^2}.
\end{align}
It should be remarked  given any point $(x,y)\in TM_0$, we have  $\nabla {\Psi}|_{(x,y)}\neq 0$. In fact, if $\nabla {\Psi}=g^{il}\frac{\pa \Psi}{\pa y^i} \frac{\pa}{\pa x^m}=0$, then $g^{im}\frac{\pa \Psi}{\pa y^i}=0 $ and due to the positive definiteness of $g^{im}$, we have $\frac{\pa \Psi}{\pa y^i}=0$.
By $1$-homogeneity of $\Psi$,   Euler's theorem asserts ${y^i}\frac{{\partial \Psi }}{{\partial {y^i}}} = \Psi$.
Hence $\Psi  = 0$, and $X$ becomes an affine vector field.  Due to the assumption of Theorem, $X$ is non-affine and therefore  $\nabla \Psi  \ne 0$.
Replacing $F^2(y)=g_y(y,y)$ in  \eqref{3002} and comparing with \eqref{3003},  yields ${\bf K}(p,y) = {k}$.
This completes the proof of Theorem \ref{Th;first}.\hspace{\stretch{1}}$\Box$\\

 \emph{\textbf{Proof of Theorem \ref{Th;second}.}}
  Let $(M,F)$ be a non-Riemannian Landsberg space and $X$  a non affine $I$-invariant projective vector field on $(M,F)$. For a Landsberg space we have
${\pounds}_{\hat{X}}{\bf{J}}=0$ and hence $\pounds_{\hat{X}}J_k=0$, where  ${J_k} = {\nabla_0I_{k}}$.
Equation \eqref{proj 8} yields
$\nabla_0f_{.k}+\Psi I_k=0$.
The equation \eqref{proj 7} and Definition \ref{Def;I,E,H} for \textbf{I}-projective vector fields follows $f_{.k}=0$. Thus $\Psi I_k=0$. If ${I_k} \ne 0$, it results that $\Psi=0$ and $X$ is an affine vector field, which is in contradiction with the hypothesis. Otherwise ${I_k}= 0$,  and by Diecke's Theorem \ref{Th;Deicke}, it turns out $F$ is Riemannian.
This completes the proof of Theorem \ref{Th;second}. \hspace{\stretch{1}}$\Box$\\

After a result in \cite{M1}, every compact Finsler space of constant ${\bf S}$-curvature is automatically of vanishing ${\bf S}$-curvature.
 It follows that every compact Einstein Randers space is of vanishing ${\bf S}$-curvature. This leads to the following corollary.
\begin{cor}
\label{COR}
On a compact Einstein Randers space of dimension $n>2$, every projective vector field is
 {\bf E}-invariant.
\end{cor}

Now we are in a position to extend the following theorem of Akbar-Zadeh for isotropic  mean Landsberg spaces.
\begin{theo}\cite{AH}
 Let $(M,F)$ be a compact Landsberg manifold and $X\in
P_{0}(M,r)$. If $R(X,X)$ is negative definite, then projective transformation corresponding to $X$ is reduced to identity. If $R(X,X)$ is non-positive definite, then
$X$ is a  Killing and horizontal covariant derivative of $X$ is zero.
\end{theo}

\emph{\textbf{Proof of Theorem \ref{Th;third}.}}
 Let $X$ be a projective vector field on the Finsler manifold $(M,F)$.
Lemma \ref{prop2} asserts that
\begin{equation}\label{1}
\nabla_0t_{ij}=2\Psi g_{ij}+y_i\Psi_{j}+y_{j}\Psi_{i},
 \end{equation}
where, ${t_{ij}} = {\pounds}_{\hat{X}}{g_{ij}}$.
Taking the vertical derivative of \eqref{1} and using
${\pounds}_{\hat{X}}E_{jk}=0$,  shows that
\begin{equation}\label{2}
\nabla_k(\nabla_jX^{i}+C^{i}_{jh}\nabla_0X^{h})+X^{h}R^{i}_{jhk}=\delta^{i}_{j}\Psi_{k}+\delta^{i}_{k}
 \Psi_{j}+C^{i}_{jk}\Psi,
 \end{equation}
see for instance \cite[p.59]{A2}. Consider a $2$-form $a(X)$  on $SM$.
\begin{equation}\label{3}
a(X):=(1/2)(\nabla_iX_{j}-\nabla_jX_{i})dx^{i}\wedge dx^{j}.
\end{equation}
Let $i_{(X)}$ be the interior product operator of $X$ and use \eqref{2} and \eqref{3} to get
\begin{equation}\label{4}
\delta(i_{(X)}a(X))=g(a(X),a(X))-[2R_{ij}X^{i}X^{j}+(n-1)X^{i}\Psi_{i}],
\end{equation}
where  $R_{ij}$ is the  trace of Cartan hh-curvature, cf. \eqref{Rik} and $\delta$ designates the co-differential operator given by \eqref{bz}.
 Recall that for a projective vector field $X$ by Lemma \eqref{la} we have
 \begin{equation}\label{5}
{\cal{L}}_{\hat{X}}\nabla_0C^{i}_{jk}=
\nabla_0({\cal{L}}_{\hat{X}}C^{i}_{jk})+\Psi C^{i}_{jk}+y^{i}C^{s}_{jk}\Psi_{s},
\end{equation}
where
where
\begin{equation}\label{6}
{\cal{L}}_{\hat{X}}C^{i}_{jk}=\dot{\nabla }_{k} (\nabla_{j} X^{i}+C^{i}_{jh}\nabla_{0} X^{h})+X^{h}P^{i}_{\ jhk}+\nabla_{0} X^{h}Q^{i}_{\ jhk}.
\end{equation}
Contracting the indices $i$ and $j$ in \eqref{6} leads
\begin{equation}\label{7}
{\pounds}_{\hat{X}}I_{k}=(\nabla_{i} X^{i}+I_{i}\nabla_{0}X^{i})_{.k})=f_{.k},
\end{equation}
 Contracting the indices $i$ and $j$ in \eqref{md}, yields
\begin{equation}\label{9}
{\pounds}_{\hat{X}}\nabla_{0} I_{k}={\pounds}_{\hat{X}} J_k =\nabla_{0} f_{.k} + \Psi I_{k}.
\end{equation}
Taking into account  $f=\nabla_{i} X^{i}+I_{i}\nabla_{0}X^{i}$, in \eqref{proj 7} by a vertical derivation of $\Psi=\frac{1}{n+1}\nabla_{0}f=\frac{1}{n+1}\nabla_{0}(\nabla_{i} X^{i}+I_{i}\nabla_{0}X^{i})$
 we get
\begin{equation}\label{ma}
(n+1)\Psi_{i}=\nabla_{i} f+\nabla_{0} f_{.i}.
\end{equation}
Let $(M,F)$ be an isotropic mean Landsberg manifold. By definition there is a function $\lambda(x,y)$ on TM such that  $\nabla_{0} I_{k}= {J_k} = \lambda (x,y){I_k}$. The Lie derivative
\begin{equation}\label{10}
{\pounds}_{\hat{X}}{J_k} ={\pounds}_{\hat{X}}\lambda(x,y)I_{k}=({\pounds}_{\hat{X}}\lambda(x,y))I_{k}
+\lambda(x,y){\pounds}_{\hat{X}}I_k.
\end{equation}
By hypothesis,  ${\pounds}_{\hat{X}}I_{k}=0$, and \eqref{10} becomes
\begin{equation}\label{11}
{\pounds}_{\hat{X}}{J_k} =({\pounds}_{\hat{X}}\lambda(x,y))I_{k}=\hat{X}.(\lambda(x,y))I_{k},
\end{equation}
 and  by means of \eqref{9} we get
\begin{equation}\label{12}
\nabla_{0} f_{.k} =\hat{X}.\lambda(x,y)I_{k} -\Psi I_{k}.
\end{equation}
Setting $h(x,y):=\hat{X}.\lambda(x,y)$, and using \eqref{ma} and \eqref{12}, leads
\begin{equation}\label{13}
(n+1)\Psi_{i}=\nabla_{i} f+hI_{i}-\Psi I_i.
\end{equation}
Convecting  $X^i$ to the both sides of \eqref{13},  results
\be\label{EQ;X^i}
(n+1)X^{i}\Psi_{i}=X^{i}\nabla_{i} f  +hI_{i}X^{i}-\Psi X^{i} I_{i}=
\nabla_{i} \left( X^{i}f\right)  -f\nabla_{i} X^{i}+h\rho-\Psi\rho,
\ee
where we place  
 $\rho:={X^i}{I_i}:=g(X,I^{*})$.
  Adding and subtracting the term $I_{i}\nabla_{0}X^{i}$ yields
\begin{align}\label{14}
(n+1)X^{i}\Psi_{i}=\nabla_{i} \left( X^{i}f\right)  -f(\nabla_{i} X^{i}+I_{i}\nabla_{0}X^{i})+f\nabla_{0} \rho +h\rho-\rho\Psi.
\end{align}
Recall that $f\nabla_{0} \rho  = \delta(f\rho V) - (n + 1)\rho \Psi$, where $\Psi=(\frac{1}{n+1})\nabla_{0} (\nabla_{i} X^{i}+I_{i}\nabla_{0}X^{i})$.
Therefore \eqref{14}  reads
\begin{align}\label{15}
(n+1)X^{i}\Psi_{i}=-\delta((X+\rho V)f)-f^{2}+h\rho-fg\rho-(n+2)\Psi\rho.
\end{align}
In order to compute ${X^i}{\psi _i}$ in \eqref{4}, we show that $\Psi\rho$ in the above equation is a divergence term on $SM$. We set
\begin{align*}
{Y_i} = g(X,l){\Psi _i} - {F^{ - 1}}g(X,l)\Psi {l_i},
\end{align*}
where $Y_i$ are the components of a vertical 1-form on $SM$.
%
 We are going to show in details that the divergence operator $\dot\delta$ of the vertical 1-form $Y_i$   on $SM$ yields
\begin{align}\label{mb}
 - \dot \delta{Y} = {X^i}{\Psi _i} - n{F^{ - 1}}g(X,l)\Psi.
\end{align}
If $Y = {Y_k}\frac{{\delta {y^k}}}{F}$ is a vertical 1-form then $\dot \delta Y$ the vertical divergence of $Y$  defined by \eqref{div.operators+2} is $\dot \delta Y = - F{g^{ij}}\frac{{\partial {Y_j}}}{{\partial {y^i}}}$.
Let us assume ${Y_i} = g(X,l){\Psi _i} - \frac{{g(X,l)}}{F}\Psi {l_i}$, where ${l^i} = \frac{{{y^i}}}{F}$,
$g(X,l) = {g_{km}}{X^k}\frac{{{y^m}}}{F} = {g_{km}}{X^k}{l^m}$. We have,
\begin{align*}
{{\dot \partial }_j}(g(X,l)) =& {{\dot \partial }_j}({g_{km}}{X^k}{l^m})= 2{C_{kmj}}{X^k}{l^m} + {g_{km}}{{\dot \partial }_j}{X^k}{l^m} + {g_{km}}{X^k}{{\dot \partial }_j}{l^m}\\ = & {g_{kl}}{X^k}(\frac{{\delta _j^m - {l^m}{l_j}}}{F}) = \frac{{{X_j} - g(X,l){l_j}}}{F}.
\end{align*}
Hence,
\begin{align*}
{l_j} = {g_{ij}}\frac{{{y^i}}}{F} = {(\frac{1}{2}{F^2})_{{.i}{.j}}}\frac{{{y^i}}}{F} = \frac{{{{(\frac{1}{2}{F^2})}_{{.j}}}}}{F} = \frac{{F{F_{{y^j}}}}}{F} = {F_{{.j}}},
\end{align*}
and we get ${l_j} = {{\dot \partial }_j}F$. As well we have
\begin{align*}
{g_{ij}} = {(\frac{1}{2}{F^2})_{{.i}{.j}}} = {F_{{.i}}}{F_{{.j}}} + F{F_{{.i}{.j}}} = {l_i}{l_j} + F{{\dot \partial }_j}{l_i}.
\end{align*}
Therefore, ${{\dot \partial }_j}{l_i} = \frac{{{g_{ij}} - {l_i}{l_j}}}{F}$. Next we compute the term $\dot \delta Y$:
\begin{align*}
\dot \delta Y =& - F{g^{ij}}{{\dot \partial }_j}(g(X,l){\Psi _i}) + F{g^{ij}}{{\dot \partial }_j}(\frac{{g(X,l)}}{F}\Psi {l_i})\\& = - F{g^{ij}}{{\dot \partial }_j}(g(X,l)){\Psi _i} - F{g^{ij}}g(X,l){\Psi _{ij}} + F{g^{ij}}{{\dot \partial }_j}(g(X,l))\frac{1}{F}\Psi {l^i}  \\ & \ \quad + F{g^{ij}}\frac{{g(X,l)}}{F}{{\dot \partial }_j}\Psi {l_i} + F{g^{ij}}\frac{{g(X,l)}}{F}{{\dot \partial }_j}{l_i}\Psi + F{g^{ij}}{{\dot \partial }_j}(\frac{1}{F})g(X,l)\Psi {l_i}.
\end{align*}

Replacing the terms ${{\dot \partial }_j}(g(X,l)) = \frac{{{X_j} - g(X,l){l_j}}}{F}$, ${{\dot \partial }_j}F = {l_j}$ and
${{\dot \partial }_j}{l_i} = \frac{{{g_{ij}} - {l_i}{l_j}}}{F}$ in the above equation yields

\begin{align*}
\dot \delta Y = &- F{g_{ij}}(\frac{{{X_j} - g(X,l){l_j}}}{F}){\Psi _i} - F{g^{ij}}g(X,l){\Psi _{ij}} + F{g^{ij}}(\frac{{{X_j} - g(X,l){l_j}}}{F})\frac{1}{F}\Psi {l_i} \\&+ F{g^{ij}}\frac{{g(X,l)}}{F}{\Psi _j}{l_i} + F{g^{ij}}\frac{{g(X,l)}}{F}\Psi {{\dot \partial }_j}{l_i} + F{g^{ij}}\frac{{( - {{\dot \partial }_j}F)}}{{{F^2}}}g(X,l)\Psi {l^i}\\ =& - {X^i}{\Psi _i} + g(X,l){\Psi _i}{l^i} - Fg(X,l){g^{ij}}{\Psi _{ij}} + {F^{ - 1}}{X^i}\Psi {l_i} + \\&{F^{ - 1}}g(X,l)\Psi ({g^{ij}}{g_{ij}} - {l^j}{l_j}) - {F^{ - 1}}g(X,l)\Psi\\ =& - {X^i}{\Psi _i} + {F^{ - 1}}g(X,l)\Psi + {F^{ - 1}}g(X,l)\Psi - 2{F^{ - 1}}g(X,l)\Psi + n{F^{ - 1}}g(X,l)\Psi\\=&
 - {X^i}{\Psi _i} + n{F^{ - 1}}g(X,l)\Psi.
\end{align*}

Therefore we have \eqref{mb}.
Similarly, if we set
${Z_i} = {F^{ - 1}}\left[ {\Psi .{X_i} - g(X,l)\Psi {l_i}} \right],$
then applying the divergence operator  yields
\begin{align}\label{mc}
- \dot \delta{Z}= n{F^{ - 1}}g(X,l)\Psi  - {X^i}{\Psi _i} - 2g(X,{I^*})\Psi,
\end{align}
where we used the notation $g(X,I^{*}):={X^i}{I_i}=\rho$.
Adding \eqref{mb} and \eqref{mc} gives
\begin{align*}
\dot \delta {(Y - Z)} = 2g(X,{I^*})\Psi=2\rho\Psi.
\end{align*}
Next we show that $h\rho$ and $fg\rho$ are divergence terms on $SM$.
Let
\begin{align*}
W_{i}:=g(X,u)\dot \delta_{i}h-F^{-1}g(X,l)h.l_{i},
\end{align*}
where $Y_{i}$ are components of a vertical 1-form. We have
\begin{equation}\label{16}
-\dot \delta{W}=X^{i}\dot \delta_{i}h-nF^{-1}g(X,l)h.
\end{equation}
If we set
\begin{align*}
U_{i}=F^{-1}h X_{i}-F^{-1}g(X,l) h l_{i},
\end{align*}
then we get
\begin{equation}\label{17}
-\dot \delta{U}=X^{i}\dot \delta_{i}h+2g(X,I^{*})h-nF^{-1}g(X,l)h.
\end{equation}
By \eqref{16} and \eqref{17}, we conclude
\begin{equation}\label{WU}
\dot \delta(W-U)=2g(X,I^{*})h=2\rho h,
\end{equation}
is a divergence term.
In the same way, one can show that $fg\rho$ is a divergence.
Using (\ref{WU}), rewrite \eqref{15}  in the following form
\begin{eqnarray}\label{XP}
\nonumber(n+1)X^{i}\Psi_{i}\!\!\!\!&=\!\!\!\!&-\delta((X+\rho
V)f)-\frac{1}{2}\dot \delta(W-U)-\frac{(n+2)}{2}\ \dot \delta (Y-Z)-f^{2}\\ \!\!\!\!&=\!\!\!\!& Div-f^{2}.
\end{eqnarray}
Replacing \eqref{XP} in \eqref{4}, yields
\begin{equation}\label{del}
\delta(i_{(X)}a(X))=g(a(X),a(X))+\frac{n-1}{n+1}\delta((X+\rho
V)f)^{2}-2R_{ij}X^{i}X^{j}+Div,
\end{equation}
where  $R_{ij}$ is the  trace of the Cartan hh-curvature.
Integrating \eqref{del} on $SM$, and omitting the divergence terms by means of Corollary \ref{Cor;Green} of Green's Theorem leads
\begin{equation*}
<a(X),a(X))>+\frac{n-1}{n+1}\underset{SM}\int\delta(X+\rho V)^{2}\eta=
2\underset{SM}\int R_{ij}X^{i}X^{j}\eta.
\end{equation*}
If $R(X,X)=R_{ij}X^{i}X^{j}<0$ is negative, since the left-hand side of the above equation is positive, the equality holds if $X=0$,  and the infinitesimal projective transformation  corresponding to $X$, is reduced to identity.\\
If $R(X,X)\leq 0$, then
$\delta(X+\rho V)=0\  \textrm{and}\  a(X)=0.$
Since $\delta(X+\rho V)=0$, $X$ is a Killing vector field  and
 is of horizontal covariant derivative null.
  This completes the proof of Theorem \ref{Th;third}. \hspace{\stretch{1}}$\Box$\\
 With the similar proof of  Theorem \ref{Th;third}  and using Definition\ref{Def;average Ricci} we have the following corollary.
\begin{cor}
Let $(M,F)$ be a compact isotropic mean Landsberg manifold
 and $X$ a projective vector field leaving invariant the
 dual trace of torsion.
 \\1. If the average Ricci curvature $\tilde{R} (X,X)$ is negative definite, then the
projective transformation corresponding to $X$ reduces to identity.
\\2. If the average Ricci curvature $\tilde{R} (X,X)$ is non-positive definite, then $X$ is Killing and its
horizontal  derivative vanishes.
\end{cor}
\textbf{ Acknowledgement.}
The second author would like to express his gratitude to the ITM (Institute Mathematics of Toulouse) where this article is partially written.

Behnaz Lajmiri$^1$, Behroz Bidabad$^{1,3}$, Mehdi Rafie-Rad$^2$, Yadollah Keshavarzi$^1$
\\ $^1$ Department of Mathematics and Computer Science Amirkabir University of Technology (Tehran Polytechnic) 424 Hafez Ave. 15914 Tehran, Iran.\\
$^2$ Department of Mathematics,
Mazandaran University, Babolsar, Iran.\\
$^3$ Institut de Mathematique de Toulouse,
Universit\'{e} Paul Sabatier, 118 route de Narbonne - F-31062 Toulouse, France.\\
    \textbf{ E-mail address:} behnaz.lajmiri@aut.ac.ir; \  bidabad@aut.ac.ir; \\
    behroz.bidabad@math.univ-toulouse.fr; \ rafie-rad@umz.ac.ir; \  y.keshavarz@aut.ac.ir.

\end{document}